\documentclass[graybox]{svmult}

\usepackage{mathptmx}       
\usepackage{helvet}         
\usepackage{courier}        
\usepackage{type1cm}        
\usepackage{makeidx}         
\usepackage{graphicx}        
\usepackage{multicol}        
\usepackage[bottom]{footmisc}

\makeindex             

\usepackage{amsmath,amssymb}
\usepackage{url}

\begin{document}

\title*{Explicit Matrices with the Restricted Isometry Property: Breaking the Square-Root Bottleneck}
\titlerunning{Explicit Matrices with the Restricted Isometry Property}
\author{Dustin G. Mixon}
\institute{Dustin G. Mixon \at Air Force Institute of Technology, Wright-Patterson Air Force Base, Ohio, USA, \email{dustin.mixon@afit.edu}}
%
%
\maketitle


\abstract{
Matrices with the restricted isometry property (RIP) are of particular interest in compressed sensing.
To date, the best known RIP matrices are constructed using random processes, while explicit constructions are notorious for performing at the ``square-root bottleneck,'' i.e., they only accept sparsity levels on the order of the square root of the number of measurements.
The only known explicit matrix which surpasses this bottleneck was constructed by Bourgain, Dilworth, Ford, Konyagin and Kutzarova in~\cite{BourgainDFKK:11}.
This chapter provides three contributions to further the groundbreaking work of Bourgain et al.:
(i) we develop an intuition for their matrix construction and underlying proof techniques; (ii) we prove a generalized version of their main result; and (iii) we apply this more general result to maximize the extent to which their matrix construction surpasses the square-root bottleneck.
}

\section{Introduction}

A matrix $\Phi$ is said to satisfy the $(K,\delta)$-\textit{restricted isometry property (RIP)} if
\begin{equation*}
(1-\delta)\|x\|^2\leq\|\Phi x\|^2\leq(1+\delta)\|x\|^2
\end{equation*}
for every $K$-sparse vector $x$.
RIP matrices are useful when compressively sensing signals which are sparse in some known orthonormal basis.
Indeed, if there is an orthogonal sparsity matrix $\Psi$ such that every signal of interest $x$ has the property that $\Psi x$ is $K$-sparse, then any such $x$ can be stably reconstructed from measurements of the form $y=Ax$ by minimizing $\|\Psi x\|_1$ subject to the measurements, provided $A\Psi^{-1}$ satisfies $(2K,\delta)$-RIP with $\delta<\sqrt{2}-1$~\cite{Candes:08}.
For sensing regimes in which measurements are costly, it is desirable to minimize the number of measurements necessary for signal reconstruction; this corresponds to the number of rows $M$ in the $M\times N$ sensing matrix $A$.
One can apply the theory of Gelfand widths to show that stable reconstruction by L1-minimization requires $K=O(M/\log(N/M))$~\cite{BaraniukDDW:08}, and random matrices show that this bound is essentially tight; indeed, $M\times N$ matrices with iid subgaussian entries satisfy $(2K,\delta)$-RIP with high probability provided $M=\Omega_\delta(K\log(N/K))$~\cite{FoucartR:13}.

Unfortunately, random matrices are not always RIP, though the failure rate vanishes asymptotically.
In applications, you might wish to verify that your randomly drawn matrix actually satisfies RIP before designing your sensing platform around that matrix, but unfortunately, this is NP-hard in general~\cite{BandeiraDMS:13}.
As such, one is forced to blindly assume that the randomly drawn matrix is RIP, and admittedly, this is a reasonable assumption considering the failure rate.
Still, this is dissatisfying from a theoretical perspective, and it motivates the construction of explicit RIP matrices:

\begin{definition}
Let $\mathrm{ExRIP}[z]$ denote the following statement:

There exists an explicit family of $M\times N$ matrices with arbitrarily large aspect ratio $N/M$ which are $(K,\delta)$-RIP with $K=\Omega(M^{z-\epsilon})$ for all $\epsilon>0$ and $\delta<\sqrt{2}-1$.
\end{definition}

Since there exist (non-explicit) matrices satisfying $z=1$ above, the goal is to prove $\mathrm{ExRIP}[1]$.
The most common way to demonstrate that an explicit matrix $\Phi$ satisfies RIP is to leverage the pairwise incoherence between the columns of $\Phi$.
Indeed, it is straightforward to prove $\mathrm{ExRIP}[1/2]$ by taking $\Phi$ to have near-optimally incoherent unit-norm columns and appealing to interpolation of operators or Gershgorin's circle theorem (e.g., see~\cite{ApplebaumHSC:09,Devore:07,FickusMT:12}).
The emergence of this ``square-root bottleneck'' compelled Tao to pose the explicit construction of RIP matrices as an open problem~\cite{Tao:07}.
Since then, only one construction has managed to break the bottleneck:
In~\cite{BourgainDFKK:11}, Bourgain, Dilworth, Ford, Konyagin and Kutzarova prove $\mathrm{ExRIP}[1/2+\epsilon_0]$ for some undisclosed $\epsilon_0>0$.
This constant has since been estimated as $\epsilon_0\approx 5.5169\times 10^{-28}$~\cite{Mixon:13a}.

Instead of estimating $\delta$ in terms of coherence, Bourgain et al.\ leverage additive combinatorics to construct $\Phi$ and to demonstrate certain cancellations in the Gram matrix $\Phi^*\Phi$.
Today (three years later), this is the only known explicit construction which breaks the square-root bottleneck, thereby leading to two natural questions:
\begin{itemize}
\item What are the proof techniques that Bourgain et al.\ applied?
\item Can we optimize the analysis to increase $\epsilon_0$?
\end{itemize}
These questions were investigated recently in a series of blog posts~\cite{Mixon:13a,Mixon:13b,Mixon:13c}, on which this chapter is based.
In the next section, we provide some preliminaries---we first cover the techniques used in~\cite{BourgainDFKK:11} to demonstrate RIP, and then we discuss some basic additive combinatorics to motivate the matrix construction.
Section~3 then describes the construction of $\Phi$, namely a subcollection of the chirps studied in~\cite{CasazzaF:06}, and discusses one method of selecting chirps (i.e., the method of Bourgain et al.).
Section~4 provides the main result, namely \textit{the BDFKK restricted isometry machine}, which says that a ``good'' selection of chirps will result in an RIP matrix construction which breaks the square-root bottleneck.
This is a generalization of the main result in~\cite{BourgainDFKK:11}, as it offers more general sufficient conditions for good chirp selection, but the proof is similar.
After generalizing the sufficient conditions, we optimize over these conditions to increase the largest known $\epsilon_0$ for which $\mathrm{ExRIP}[1/2+\epsilon_0]$ holds: 
\begin{equation*}
\epsilon_0\approx 4.4466\times 10^{-24}.
\end{equation*}
Of course, any improvement to the chirp selection method will further increase this constant, and hopefully, the BDFKK restricted isometry machine and overall intuition provided in this chapter will foster such progress.
Section~5 contains the proofs of certain technical lemmas that are used to prove the main result.

\section{Preliminaries}

The goal of this section is to provide some intuition for the main ideas in~\cite{BourgainDFKK:11}.
We first explain the overall proof technique for demonstrating RIP (this is the vehicle for breaking the square-root bottleneck), and then we introduce some basic ideas from additive combinatorics.

\subsection{The Big-Picture Techniques}

Before explaining how Bourgain et al.\ broke the square-root bottleneck, let's briefly discuss the more common, coherence-based technique to demonstrate RIP.
Let $\Phi_\mathcal{K}$ denote the submatrix of $\Phi$ whose columns are indexed by $\mathcal{K}\subseteq\{1,\ldots,N\}$.
Then $(K,\delta)$-RIP equivalently states that, for every $\mathcal{K}$ of size $K$, the eigenvalues of $\Phi_\mathcal{K}^*\Phi_\mathcal{K}$ lie in $[1-\delta,1+\delta]$.
As such, we can prove that a matrix is RIP by approximating eigenvalues.
To this end, if we assume the columns of $\Phi$ have unit norm, and if we let $\mu$ denote the largest off-diagonal entry of $\Phi^*\Phi$ in absolute value (this is the worst-case coherence of the columns of $\Phi$), then the Gershgorin circle theorem implies that $\Phi$ is $(K,(K-1)\mu)$-RIP.
Unfortunately, the coherence can't be too small, due to the Welch bound~\cite{Welch:74}:
\begin{equation*}
\mu\geq\sqrt{\frac{N-M}{M(N-1)}},
\end{equation*}
which is $\Omega(M^{-1/2})$ provided $N\geq cM$ for some $c>1$.
Thus, to get $(K-1)\mu=\delta<1/2$, we require $K<1/(2\mu)+1=O(M^{1/2})$.
This is much smaller than the random RIP constructions which instead take $K=O(M^{1-\epsilon})$ for all $\epsilon>0$, thereby revealing the shortcoming of the Gershgorin technique.

Now let's discuss the alternative techniques that Bourgain et al. use.
The main idea is to convert the RIP statement, which concerns all $K$-sparse vectors simultaneously, into a statement about finitely many vectors:

\begin{definition}[flat RIP]
We say $\Phi=[\varphi_1\cdots\varphi_N]$ satisfies $(K,\theta)$-\textit{flat RIP} if for every disjoint $I,J\subseteq\{1,\ldots,N\}$ of size $\leq K$,
\begin{equation*}
\bigg|\bigg\langle\sum_{i\in I}\varphi_i,\sum_{j\in J}\varphi_j\bigg\rangle\bigg|\leq \theta\sqrt{|I||J|}.
\end{equation*}
\end{definition}

\begin{lemma}[essentially Lemma~3 in~\cite{BourgainDFKK:11}, cf. Theorem 13 in~\cite{BandeiraFMW:13}]
\label{lemma.a}
If $\Phi$ has $(K,\theta)$-flat RIP and unit-norm columns, then $\Phi$ has $(K,150\theta\log K)$-RIP.
\end{lemma}

Unlike the coherence argument, flat RIP doesn't lead to much loss in $K$.
In particular, \cite{BandeiraFMW:13} shows that random matrices satisfy $(K,\theta)$-flat RIP with $\theta=O(\delta/\log K)$ when $M=\Omega((K/\delta^2)\log^2 K\log N)$.
As such, it makes sense that flat RIP would be a vehicle to break the square-root bottleneck.
However, in practice, it's difficult to control both the left- and right-hand sides of the flat RIP inequality---it would be much easier if we only had to worry about getting cancellations, and not getting different levels of cancellation for different-sized subsets. This leads to the following:

\begin{definition}[weak flat RIP]
\label{defn.weak flat RIP}
We say $\Phi=[\varphi_1\cdots\varphi_N]$ satisfies $(K,\theta')$-\textit{weak flat RIP} if for every disjoint $I,J\subseteq\{1,\ldots,N\}$ of size $\leq K$,
\begin{equation*}
\bigg|\bigg\langle\sum_{i\in I}\varphi_i,\sum_{j\in J}\varphi_j\bigg\rangle\bigg|\leq \theta' K.
\end{equation*}
\end{definition}

\begin{lemma}[essentially Lemma~1 in~\cite{BourgainDFKK:11}]
\label{lemma.b}
If $\Phi$ has $(K,\theta')$-weak flat RIP and worst-case coherence $\mu\leq1/K$, then $\Phi$ has $(K,\sqrt{\theta'})$-flat RIP.
\end{lemma}

\begin{proof}
By the triangle inequality, we have
\begin{equation*}
\bigg|\bigg\langle\sum_{i\in I}\varphi_i,\sum_{j\in J}\varphi_j\bigg\rangle\bigg|\leq\sum_{i\in I}\sum_{j\in J}|\langle\varphi_i,\varphi_j\rangle|\leq |I||J|\mu\leq |I||J|/K.
\end{equation*}
Since $\Phi$ also has weak flat RIP, we then have
\begin{equation*}
\bigg|\bigg\langle\sum_{i\in I}\varphi_i,\sum_{j\in J}\varphi_j\bigg\rangle\bigg|\leq\min\{\theta'K,|I||J|/K\}\leq\sqrt{\theta'|I||J|}.
\quad\qed
\end{equation*}
\end{proof}

Unfortunately, this coherence requirement puts $K$ back in the square-root bottleneck, since $\mu\leq1/K$ is equivalent to $K\leq1/\mu=O(M^{1/2})$.
To rectify this, Bourgain et al. use a trick in which a modest $K$ with tiny $\delta$ can be converted to a large $K$ with modest $\delta$:

\begin{lemma}[buried in Lemma 3 in~\cite{BourgainDFKK:11}, cf. Theorem 1 in~\cite{KoiranZ:12}]
\label{lemma.c}
If $\Phi$ has $(K,\delta)$-RIP, then $\Phi$ has $(sK,2s\delta)$-RIP for all $s\geq1$.
\end{lemma}

In~\cite{KoiranZ:12}, this trick is used to get RIP results for larger $K$ when testing RIP for smaller $K$.
For the explicit RIP matrix problem, we are stuck with proving how small $\delta$ is when $K$ on the order of $M^{1/2}$.
Note that this trick will inherently exhibit some loss in $K$.
Assuming the best possible scaling for all $N$, $K$ and $\delta$ is $M=\Theta((K/\delta^2)\log(N/K))$, then if $N=\mathrm{poly}(M)$, you can get $(M^{1/2},\delta)$-RIP only if $\delta=\Omega((\log^{1/2}M)/M^{1/4})$.
In this best-case scenario, you would want to pick $s=M^{1/4-\epsilon}$ for some $\epsilon>0$ and apply Lemma~\ref{lemma.c} to get $K=O(M^{3/4-\epsilon})$.
In some sense, this is another manifestation of the square-root bottleneck, but it would still be a huge achievement to saturate this bound.

\subsection{A Brief Introduction to Additive Combinatorics}

In this subsection, we briefly detail some key ideas from additive combinatorics; the reader is encouraged to see~\cite{TaoV:06} for a more complete introduction.
Given an additive group $G$ and finite sets $A,B\subseteq G$, we can define the sumset
\begin{equation*}
A+B:=\{a+b:a\in A,~b\in B\},
\end{equation*}
the difference set
\begin{equation*}
A-B:=\{a-b:a\in A,~b\in B\},
\end{equation*}
and the additive energy
\begin{equation*}
E(A,B):=\#\big\{(a_1,a_2,b_1,b_2)\in A^2\times B^2:a_1+b_1=a_2+b_2\big\}.
\end{equation*}
These definitions are useful in quantifying the additive structure of a set.
In particular, consider the following:

\begin{lemma}
A nonempty subset $A$ of some additive group $G$ satisfies the following inequalities:
\begin{description}[(iii)]
\item[(i)] $|A+A|\geq|A|$
\item[(ii)] $|A-A|\geq|A|$
\item[(iii)] $E(A,A)\leq|A|^3$
\end{description}
with equality precisely when $A$ is a translate of some subgroup of $G$.
\end{lemma}

\begin{proof}
For (i), pick $a\in A$.
Then $|A+A|\geq|A+a|=|A|$.
Considering
\begin{equation*}
A+A=\bigcup_{a\in A}(A+a),
\end{equation*}
we have equality in (i) precisely when $A+A=A+a$ for every $a\in A$.
Equivalently, given $a_0\in A$, then for every $a\in A$, addition by $a-a_0$ permutes the members of $A+a_0$.
This is further equivalent to the following:
Given $a_0\in A$, then for every $a\in A$, addition by $a-a_0$ permutes the members of $A-a_0$.
It certainly suffices for $H:=A-a_0$ to be a group, and it is a simple exercise to verify that this is also necessary.

The proof for (ii) is similar.

For (iii), we note that
\begin{equation*}
E(A,A)=\#\big\{(a,b,c)\in A^3:a+b-c\in A\big\}\leq|A|^3,
\end{equation*}
with equality precisely when $A$ has the property that $a+b-c\in A$ for every $a,b,c\in A$.
Again, it clearly suffices for $A-a_0$ to be a group, and necessity is a simple exercise.
\qed
\end{proof}

The notion of additive structure is somewhat intuitive.
You should think of a translate of a subgroup as having maximal additive structure.
When the bounds (i), (ii) and (iii) are close to being achieved by $A$ (e.g., $A$ is an arithmetic progression), you should think of $A$ as having a lot of additive structure.
Interestingly, while there are different measures of additive structure (e.g., $|A-A|$ and $E(A,A)$), they often exhibit certain relationships (perhaps not surprisingly).
The following is an example of such a relationship which is used throughout the paper by Bourgain et al.~\cite{BourgainDFKK:11}:

\begin{lemma}[Corollary 1 in~\cite{BourgainDFKK:11}]
\label{corollary.1}
If $E(A,A)\geq|A|^3/K$, then there exists a set $A'\subseteq A$ such that $|A'|\geq|A|/(20K)$ and $|A'-A'|\leq10^7K^9|A|$.
\end{lemma}

In words, a set with a lot of additive energy necessarily has a large subset with a small difference set.
This is proved using a version of the Balog--Szemeredi--Gowers lemma~\cite{BourgainG:09}.

If translates of subgroups have maximal additive structure, then which sets have minimal additive structure?
It turns out that random sets tend to (nearly) have this property, and one way to detect low additive structure is Fourier bias:
\begin{equation*}
\|A\|_u:=\max_{\substack{\theta\in G\\\theta\neq0}}|\widehat{1_A}(\theta)|,
\end{equation*}
where the Fourier transform ($\hat{\cdot}$) used here has a $1/|G|$ factor in front (it is not unitary).
For example, if $G=\mathbb{Z}/n\mathbb{Z}$, we take
\begin{equation*}
\hat{f}(\xi)
:=\frac{1}{|G|}\sum_{x\in G}f(x)e^{-2\pi ix\xi/n}.
\end{equation*}
Interestingly, $\|A\|_u$ captures how far $E(A,A)$ is from its minimal value $|A|^4/|G|$:

\begin{lemma}
\label{lemma.fourier bias vs energy}
For any subset $A$ of a finite additive group $G$, we have
\begin{description}[(ii)]
\item[(i)] $E(A,A)\geq\frac{|A|^4}{|G|}$, and
\item[(ii)] $\|A\|_u^4\leq\frac{1}{|G|^3}\Big(E(A,A)-\frac{|A|^4}{|G|}\Big)\leq\frac{|A|}{|G|}\|A\|_u^2$.
\end{description}
\end{lemma}

\begin{proof}
Define $\lambda_x:=\#\{(a,a')\in A^2:a-a'=x\}$. Then (i) follows from Cauchy--Schwarz:
\begin{equation*}
|A|^2=\sum_{x\in G}\lambda_x\leq|G|^{1/2}\|\lambda\|_2=\big(|G|E(A,A)\big)^{1/2}.
\end{equation*}
We will prove (ii) assuming $G=\mathbb{Z}/n\mathbb{Z}$, but the proof generalizes. Denote $e_n(x):=e^{2\pi ix/n}$. For the left-hand inequality, we consider
\begin{equation*}
\sum_{\theta\in G}|\widehat{1_A}(\theta)|^4=\sum_{\theta\in G}\bigg|\frac{1}{|G|}\sum_{a\in A}e_n(-\theta a)\bigg|^4=\frac{1}{|G|^4}\sum_{\theta\in G}\bigg|\sum_{x\in G}\lambda_xe_n(-\theta x)\bigg|^2,
\end{equation*}
where the last step is by expanding $|w|^2=w\overline{w}$. Then Parseval's identity simplifies this to $\frac{1}{|G|^3}\|\lambda\|_2^2=\frac{1}{|G|^3}E(A,A)$. We use this to bound $\|A\|_u^4$:
\begin{equation*}
\|A\|_u^4=\max_{\substack{\theta\in G\\\theta\neq0}}|\widehat{1_A}(\theta)|^4\leq\sum_{\substack{\theta\in G\\\theta\neq0}}|\widehat{1_A}(\theta)|^4=\frac{1}{|G|^3}E(A,A)-\frac{|A|^4}{|G|^4}.
\end{equation*}
For the right-hand inequality, we apply Parseval's identity:
\begin{equation*}
E(A,A)=\sum_{x\in G}\lambda_x^2=\frac{1}{|G|}\sum_{\theta\in G}\bigg|\sum_{x\in G}\lambda_xe_n(-\theta x)\bigg|^2=\frac{|A|^4}{|G|}+\frac{1}{|G|}\sum_{\substack{\theta\in G\\\theta\neq0}}\bigg|\sum_{x\in G}\lambda_xe_n(-\theta x)\bigg|^2
\end{equation*}
From here, we apply the expansion $|w|^2=w\overline{w}$
\begin{equation*}
\bigg|\sum_{a\in A}e_n(-\theta a)\bigg|^2=\sum_{x\in G}\lambda_xe_n(-\theta x)
\end{equation*}
to continue:
\begin{equation*}
\sum_{\substack{\theta\in G\\\theta\neq0}}\bigg|\sum_{x\in G}\lambda_xe_n(-\theta x)\bigg|^2=\sum_{\substack{\theta\in G\\\theta\neq0}}\bigg|\sum_{a\in A}e_n(-\theta a)\bigg|^4\leq\sum_{\substack{\theta\in G\\\theta\neq0}}\big(|G|\|A\|_u\big)^2\bigg|\sum_{a\in A}e_n(-\theta a)\bigg|^2.
\end{equation*}
Applying Parseval's identity then gives
\begin{equation*}
E(A,A)\leq\frac{|A|^4}{|G|}+\big(|G|\|A\|_u\big)^2\cdot\frac{1}{|G|}\sum_{\theta\in G}\bigg|\sum_{a\in A}e_n(-\theta a)\bigg|^2=\frac{|A|^4}{|G|}+|G|^2\|A\|_u^2|A|,
\end{equation*}
which is a rearrangement of the right-hand inequality.
\qed
\end{proof}

\section{The Matrix Construction}

This section combines ideas from the previous section to introduce the matrix construction used by Bourgain et al.~\cite{BourgainDFKK:11} to break the square-root bottleneck.
The main idea is to construct a Gram matrix $\Phi^*\Phi$ whose entries exhibit cancellations for weak flat RIP (see Definition~\ref{defn.weak flat RIP}).
By Lemma~\ref{lemma.fourier bias vs energy}, we can control cancellations of complex exponentials
\begin{equation*}
\bigg|\sum_{a\in A}e_n(\theta a)\bigg|\leq n\|A\|_u,\qquad\theta\neq0
\end{equation*}
in terms of the additive energy of the index set $A\subseteq\mathbb{Z}/n\mathbb{Z}$; recall that $e_n(x):=e^{2\pi i x/n}$.
This motivates us to pursue a Gram matrix whose entries are complex exponentials.
To this end, consider the following vector:
\begin{equation*}
u_{a,b}:=\frac{1}{\sqrt{p}}\Big(e_p(ax^2+bx)\Big)_{x\in\mathbb{F}_p},
\end{equation*}
where $p$ is prime and $\mathbb{F}_p$ denotes the field of size $p$. Such vectors are called \textit{chirps}, and they are used in a variety of applications including radar. Here, we are mostly interested in the form of their inner products. If $a_1=a_2$, then $\langle u_{a_1,b_1},u_{a_2,b_2}\rangle=\delta_{b_1,b_2}$ by the geometric sum formula. Otherwise, the inner product is more interesting:
\begin{equation*}
\langle u_{a_1,b_1},u_{a_2,b_2}\rangle=\frac{1}{p}\sum_{x\in\mathbb{F}_p}e_p\Big((a_1-a_2)x^2+(b_1-b_2)x\Big).
\end{equation*}
Since $a_1-a_2\neq0$, we can complete the square in the exponent, and changing variables to $y:=x+(b_1-b_2)/(2(a_1-a_2))$ gives
\begin{equation*}
\langle u_{a_1,b_1},u_{a_2,b_2}\rangle=\frac{1}{p}e_p\bigg(-\frac{(b_1-b_2)^2}{4(a_1-a_2)}\bigg)\sum_{y\in\mathbb{F}_p}e_p\Big((a_1-a_2)y^2\Big).
\end{equation*}
Finally, this can be simplified using a quadratic Gauss sum formula:
\begin{equation*}
\langle u_{a_1,b_1},u_{a_2,b_2}\rangle=\frac{\sigma_p}{\sqrt{p}}\bigg(\frac{a_1-a_2}{p}\bigg)e_p\bigg(-\frac{(b_1-b_2)^2}{4(a_1-a_2)}\bigg),
\end{equation*}
where $\sigma_p$ is $1$ or $i$ (depending on whether $p$ is $1$ or $3\bmod 4$) and $(\frac{a_1-a_2}{p})$ is a Legendre symbol, taking value $\pm1$ depending on whether $a_1-a_2$ is a perfect square $\bmod p$. Modulo these factors, the above inner product is a complex exponential, and since we want these in our Gram matrix $\Phi^*\Phi$, we will take $\Phi$ to have columns of the form $u_{a,b}$---in fact, the columns will be $\{u_{a,b}\}_{(a,b)\in\mathcal{A}\times\mathcal{B}}$ for some well-designed sets $\mathcal{A},\mathcal{B}\subseteq\mathbb{F}_p$.

For weak flat RIP, we want to bound the following quantity for every $\Omega_1,\Omega_2\subseteq\mathcal{A}\times\mathcal{B}$ with $|\Omega_1|,|\Omega_2|\leq\sqrt{p}$:
\begin{equation*}
\bigg|\bigg\langle\sum_{(a_1,b_1)\in\Omega_1}u_{a_1,b_1},\sum_{(a_2,b_2)\in\Omega_2}u_{a_2,b_2}\bigg\rangle\bigg|.
\end{equation*}
For $i=1,2$, define
\begin{equation*}
A_i:=\{a:\exists b\mbox{ s.t. }(a,b)\in\Omega_i\}\quad\mbox{and}\quad\Omega_i(a):=\{b:(a,b)\in\Omega_i\}.
\end{equation*}
These provide an alternate expression for the quantity of interest:
\begin{align*}
\bigg|\sum_{(a_1,b_1)\in\Omega_1}\sum_{(a_2,b_2)\in\Omega_2}\langle u_{a_1,b_1},u_{a_2,b_2}\rangle\bigg|
&=\bigg|\sum_{\substack{a_1\in A_1\\a_2\in A_2}}\sum_{\substack{b_1\in\Omega_1(a_1)\\b_2\in\Omega_2(a_2)}}\langle u_{a_1,b_1},u_{a_2,b_2}\rangle\bigg|,\\
&\leq\sum_{\substack{a_1\in A_1\\a_2\in A_2}}\bigg|\sum_{\substack{b_1\in\Omega_1(a_1)\\b_2\in\Omega_2(a_2)}}\langle u_{a_1,b_1},u_{a_2,b_2}\rangle\bigg|\\
&=\frac{1}{\sqrt{p}}\sum_{\substack{a_1\in A_1\\a_2\in A_2}}\bigg|\sum_{\substack{b_1\in\Omega_1(a_1)\\b_2\in\Omega_2(a_2)}}e_p\bigg(-\frac{(b_1-b_2)^2}{4(a_1-a_2)}\bigg)\bigg|.
\end{align*}
Pleasingly, it now suffices to bound a sum of complex exponentials, which we feel equipped to do using additive combinatorics. The following lemma does precisely this (it can be viewed as an analog of Lemma~\ref{lemma.fourier bias vs energy}).

\begin{lemma}[Lemma~9 in~\cite{BourgainDFKK:11}]
\label{lemma.9}
For every $\theta\in\mathbb{F}_p^*$ and $B_1,B_2\subseteq\mathbb{F}_p$, we have
\begin{equation*}
\bigg|\sum_{\substack{b_1\in B_1\\b_2\in B_2}}e_p\Big(\theta(b_1-b_2)^2\Big)\bigg|\leq|B_1|^{1/2}E(B_1,B_1)^{1/8}|B_2|^{1/2}E(B_2,B_2)^{1/8}p^{1/8}.
\end{equation*}
\end{lemma}

\begin{proof}
First, Cauchy--Schwarz gives
\begin{align*}
\bigg|\sum_{\substack{b_1\in B_1\\b_2\in B_2}}e_p\Big(\theta(b_1-b_2)^2\Big)\bigg|^2
&=\bigg|\sum_{b_1\in B_1}1\cdot\sum_{b_2\in B_2}e_p\Big(\theta(b_1-b_2)^2\Big)\bigg|^2\\
&\leq|B_1|\sum_{b_1\in B_1}\bigg|\sum_{b_2\in B_2}e_p\Big(\theta(b_1-b_2)^2\Big)\bigg|^2.
\end{align*}
Expanding $|w|^2=w\overline{w}$ and rearranging then gives an alternate expression for the right-hand side:
\begin{equation*}
|B_1|\sum_{b_2,b_2'\in B_2}e_p\Big(\theta(b_2^2-(b_2')^2)\Big)\overline{\sum_{b_1\in B_1}e_p\Big(\theta(2b_1(b_2-b_2'))\Big)}.
\end{equation*}
Applying Cauchy--Schwarz again, we then have
\begin{equation*}
\bigg|\sum_{\substack{b_1\in B_1\\b_2\in B_2}}e_p\Big(\theta(b_1-b_2)^2\Big)\bigg|^4\leq|B_1|^2|B_2|^2\sum_{b_2,b_2'\in B_2}\bigg|\sum_{b_1\in B_1}e_p\Big(\theta(2b_1(b_2-b_2'))\Big)\bigg|^2,
\end{equation*}
and expanding $|w|^2=w\overline{w}$ this time gives
\begin{equation*}
|B_1|^2|B_2|^2\sum_{\substack{b_1,b_1'\in B_1\\b_2,b_2'\in B_2}}e_p\Big(2\theta(b_1-b_1')(b_2-b_2')\Big).
\end{equation*}
At this point, it is convenient to change variables, namely, $x=b_1-b_1'$ and $y=b_2-b_2'$:
\begin{equation}
\label{eq.part of lemma 9}
\bigg|\sum_{\substack{b_1\in B_1\\b_2\in B_2}}e_p\Big(\theta(b_1-b_2)^2\Big)\bigg|^4\leq |B_1|^2|B_2|^2\sum_{x,y\in\mathbb{F}_p}\lambda_x\mu_ye_p(2\theta xy),
\end{equation}
where $\lambda_x:=\#\{(b_1,b_1')\in B_1^2:b_1-b_1'=x\}$ and similarly for $\mu_y$ in terms of $B_2$. We now apply Cauchy--Schwarz again to bound the sum in \eqref{eq.part of lemma 9}:
\begin{equation*}
\bigg|\sum_{x\in\mathbb{F}_p}\lambda_x\sum_{y\in\mathbb{F}_p}\mu_ye_p(2\theta xy)\bigg|^2\leq\|\lambda\|_2^2\sum_{x\in\mathbb{F}_p}\bigg|\sum_{y\in\mathbb{F}_p}\mu_ye_p(2\theta xy)\bigg|^2,
\end{equation*}
and changing variables $x':=-2\theta x$ (this change is invertible since $\theta\neq0$), we see that the right-hand side is a sum of squares of the Fourier coefficients of $\mu$. As such, Parseval's identity gives the following simplification:
\begin{equation*}
\bigg|\sum_{x,y\in\mathbb{F}_p}\lambda_x\mu_ye_p(2\theta xy)\bigg|^2\leq p\|\lambda\|_2^2\|\mu\|_2^2=pE(B_1,B_1)E(B_2,B_2).
\end{equation*}
Applying this bound to \eqref{eq.part of lemma 9} gives the result.
\qed
\end{proof}

\subsection{How to Construct $\mathcal{B}$}

Lemma~\ref{lemma.9} enables us to prove weak-flat-RIP-type cancellations in cases where $\Omega_1(a_1),\Omega_2(a_2)\subseteq\mathcal{B}$ both lack additive structure.
Indeed, the method of \cite{BourgainDFKK:11} is to do precisely this, and the remaining cases (where either $\Omega_1(a_1)$ or $\Omega_2(a_2)$ has more additive structure) will find cancellations by accounting for the dilation weights $1/(a_1-a_2)$.
Overall, we will be very close to proving that $\Phi$ is RIP if most subsets of $\mathcal{B}$ lack additive structure.
To this end, Bourgain et al.~\cite{BourgainDFKK:11} actually prove something much stronger:
They design $\mathcal{B}$ in such a way that all sufficiently large subsets have low additive structure.
The following theorem is the first step in the design:

\begin{theorem}[Theorem~5 in~\cite{BourgainDFKK:11}]
\label{theorem.5}
Fix $r,M\in\mathbb{N}$, $M\geq2$, and define the cube $\mathcal{C}:=\{0,\ldots,M-1\}^r\subseteq\mathbb{Z}^r$. Let $\tau$ denote the solution to the equation
\begin{equation*}
\Big(\frac{1}{M}\Big)^{2\tau}+\Big(\frac{M-1}{M}\Big)^{\tau}=1.
\end{equation*}
Then for any subsets $A,B\subseteq\mathcal{C}$, we have
\begin{equation*}
|A+B|\geq\big(|A||B|\big)^\tau.
\end{equation*}
\end{theorem}

As a consequence of this theorem (taking $A=B$), we have $|A+A|\geq|A|^{2\tau}$ for every $A\subseteq\mathcal{C}$, and since $\tau>1/2$, this means that large subsets $A$ have $|A+A|\gg|A|$, indicating low additive structure.
However, $\mathcal{C}$ is a subset of the group $\mathbb{Z}^r$, whereas we need to construct a subset $\mathcal{B}$ of $\mathbb{F}_p$.
The trick here is to pick $\mathcal{B}$ so that it inherits the additive structure of $\mathcal{C}$, and we use a \textit{Freiman isomorphism} to accomplish this.
In particular, we want a mapping $\varphi\colon\mathcal{C}\rightarrow\mathbb{F}_p$ such that $c_1+c_2=c_3+c_4$ if and only if $\varphi(c_1)+\varphi(c_2)=\varphi(c_3)+\varphi(c_4)$, and we will take $\mathcal{B}:=\varphi(\mathcal{C})$---this is what it means for $\mathcal{C}$ and $\mathcal{B}$ to be Freiman isomorphic, and it's easy to see that Freiman isomorphic sets have the same sized sumsets, difference sets and additive energy.
In this case, it suffices to take
\begin{equation}
\label{eq.construction of b}
\mathcal{B}:=\bigg\{\sum_{j=1}^{r}x_j(2M)^{j-1}:x_1,\ldots,x_r\in\{0,\ldots,M-1\}\bigg\}.
\end{equation}
Indeed, the $2M$-ary expansion of $b_1,b_2\in\mathcal{B}$ reveals the $c_1,c_2\in\mathcal{C}$ such that $\varphi(c_1)=b_1$ and $\varphi(c_2)=b_2$.
Also, adding $b_1$ and $b_2$ incurs no carries, so the expansion of $b_1+b_2$ reveals $c_1+c_2$ (even when $c_1+c_2\not\in\mathcal{C}$).

We already know that large subsets of $\mathcal{C}$ (and $\mathcal{B}$) exhibit low additive structure, but the above theorem only gives this in terms of the sumset, whereas Lemma~\ref{lemma.9} requires low additive structure in terms of additive energy.
As such, we will first convert the above theorem into a statement about difference sets, and then apply Lemma~\ref{corollary.1} to further convert it in terms of additive energy:

\begin{corollary}[essentially Corollary~3 in~\cite{BourgainDFKK:11}]
Fix $r$, $M$ and $\tau$ according to Theorem~\ref{theorem.5}, take $\mathcal{B}$ as defined in \eqref{eq.construction of b}, and pick $s$ and $t$ such that $(2\tau-1)s\geq t$.
Then every subset $B\subseteq\mathcal{B}$ such that $|B|>p^s$ satisfies $|B-B|>p^t|B|$.
\end{corollary}

\begin{proof}
First note that $-B$ is a translate of some other set $B'\subseteq\mathcal{B}$.
Explicitly, if $b_0=\sum_{j=1}^r(M-1)(2M)^{j-1}$, then we can take $B':=b_0-B$.
As such, Theorem~\ref{theorem.5} gives
\begin{equation*}
|B-B|=|B+B'|\geq|B|^{2\tau}=|B|^{2\tau-1}|B|>p^{(2\tau-1)s}|B|\geq p^t|B|.
\quad
\qed
\end{equation*}
\end{proof}

\begin{corollary}[essentially Corollary~4 in~\cite{BourgainDFKK:11}]
\label{corollary.4}
Fix $r$, $M$ and $\tau$ according to Theorem~\ref{theorem.5}, take $\mathcal{B}$ as defined in \eqref{eq.construction of b}, and pick $\gamma$ and $\ell$ such that $(2\tau-1)(\ell-\gamma)\geq 10\gamma$.
Then for every $\epsilon>0$, there exists $P$ such that for every $p\geq P$, every subset $S\subseteq\mathcal{B}$ with $|S|>p^\ell$ satisfies $E(S,S)<p^{-\gamma+\epsilon}|S|^3$.
\end{corollary}

\begin{proof}
Suppose to the contrary that there exists $\epsilon>0$ such that there are arbitrarily large $p$ for which there is a subset $S\subseteq\mathcal{B}$ with $|S|>p^\ell$ and $E(S,S)\geq p^{-\gamma+\epsilon}|S|^3$.
Writing $E(S,S)=|S|^3/K$, then $K\leq p^{\gamma-\epsilon}$. By Lemma~\ref{corollary.1}, there exists $B\subseteq S$ such that, for sufficiently large $p$,
\begin{equation*}
|B|\geq|S|/(20K)>\frac{1}{20}p^{\ell-\gamma+\epsilon}>p^{\ell-\gamma},
\end{equation*}
and
\begin{equation*}
|B-B|\leq10^7K^9|S|\leq10^7K^9(20K|B|)\leq10^7\cdot20\cdot p^{10(\gamma-\epsilon)}|B|<p^{10\gamma}|B|.
\end{equation*}
However, this contradicts the previous corollary with $s=\ell-\gamma$ and $t=10\gamma$.
\qed
\end{proof}

Notice that we can weaken our requirements on $\gamma$ and $\ell$ if we had a version of Lemma~\ref{corollary.1} with a smaller exponent on $K$.
This exponent comes from a version of the Balog--Szemeredi--Gowers lemma (Lemma~6 in~\cite{BourgainDFKK:11}), which follows from the proof of Lemma~2.2 in~\cite{BourgainG:09}.
(Specifically, take $A=B$, and you need to change $A-_EB$ to $A+_EB$, but this change doesn't affect the proof.)
Bourgain et al. indicate that it would be desirable to prove a better version of this lemma, but it is unclear how easy that would be.

\subsection{How to Construct $\mathcal{A}$}

The previous subsection showed how to construct $\mathcal{B}$ so as to ensure that all sufficiently large subsets have low additive structure.
By Lemma~\ref{lemma.9}, this in turn ensures that $\Phi$ exhibits weak-flat-RIP-type cancellations for most $\Omega_1(a_1),\Omega_2(a_2)\subseteq\mathcal{B}$.
For the remaining cases, $\Phi$ must exhibit weak-flat-RIP-type cancellations by somehow leveraging properties of $\mathcal{A}$.

The next section gives the main result, which requires a subset $\mathcal{A}=\mathcal{A}(p)$ of $\mathbb{F}_p$ for which there exists an even number $m$ as well as an $\alpha>0$ (both independent of $p$) such that the following two properties hold:
\begin{description}[(ii)]
\item[(i)]
$\Omega(p^\alpha)\leq|\mathcal{A}(p)|\leq p^\alpha$.
\item[(ii)]
For each $a\in\mathcal{A}$, then $a_1,\ldots,a_{2m}\in\mathcal{A}\setminus\{a\}$ satisfies
\begin{equation}
\label{eq.2.4}
\sum_{j=1}^m\frac{1}{a-a_j}=\sum_{j=m+1}^{2m}\frac{1}{a-a_j}
\end{equation}
only if $(a_1,\ldots,a_m)$ and $(a_{m+1},\ldots,a_{2m})$ are permutations of each other. Here, division (and addition) is taken in the field $\mathbb{F}_p$.
\end{description}

Unfortunately, these requirements on $\mathcal{A}$ lead to very little intuition compared to our current understanding of $\mathcal{B}$.
Regardless, we will continue by considering how Bourgain et al. constructs $\mathcal{A}$.
The following lemma describes their construction and makes a slight improvement to the value of $\alpha$ chosen in~\cite{BourgainDFKK:11}:

\begin{lemma}
\label{lemma.construction of a}
Pick $\epsilon>0$ and take $L:=\lfloor p^{1/2m(4m-1)}\rfloor$ and $U:=\lfloor L^{4m-1}\rfloor$. Then
\begin{equation*}
\mathcal{A}:=\{x^2+Ux:1\leq x\leq L\}
\end{equation*}
satisfies (i) and (ii) above if we take
\begin{equation*}
\alpha=\frac{1}{2m(4m-1)}
\end{equation*}
and $p$ is a sufficiently large prime.
\end{lemma}

\begin{proof}
One may quickly verify (i).
For (ii), we claim it suffices to show that for any $n\in\{1,\ldots,2m\}$, any distinct $x,x_1,\ldots,x_n\in\{1,\ldots,L\}$, and any nonzero integers $\lambda_1,\ldots,\lambda_n$ such that $|\lambda_1|+\cdots+|\lambda_n|\leq 2m$, then
\begin{equation}
\label{eq.A.1}
V=\sum_{j=1}^n\frac{\lambda_j}{(x-x_j)(x+x_j+U)}
\end{equation}
is nonzero (in $\mathbb{F}_p$).
To see this, define $a:=x^2+Ux$ and $a_j:=x_j^2+Ux_j$. Then
\begin{equation*}
V=\sum_{j=1}^n\frac{\lambda_j}{a-a_j}
\end{equation*}
As such, if (ii) fails to hold, then subtracting the right-hand side of \eqref{eq.2.4} from the left-hand side produces an example of $V$ which is zero, violating our statement.
Thus, our statement implies (ii) by the contrapositive.

We now seek to prove our statement.
To this end, define $D_1:=\prod_{j=1}^n(x-x_j)$ and $D_2:=\prod_{j=1}^n(x+x_j+U)$.
Note that \eqref{eq.A.1} being nonzero in $\mathbb{F}_p$ is equivalent to having $p$ not divide $D_1D_2V$ as an integer.
To prove this, we will show that
\begin{description}[(b)]
\item[(a)]
$p$ does not divide $D_1$,
\item[(b)]
$D_2V$ is nonzero, and
\item[(c)]
$|D_2V|<p$.
\end{description}
Indeed, (b) and (c) together imply that $p$ does not divide $D_2V$, which combined with (a) implies that $p$ does not divide $D_1D_2V$.

For (a), we have $D_1\neq0$ since $x$ and the $x_j$'s are distinct by assumption, and since $x$ and each $x_j$ has size at most $L$, we also have $|D_1|\leq L^{2m}$.
To complete the proof of (a), it then suffices to have
\begin{equation}
\label{eq.a'}
L^{2m}<p,
\end{equation}
which we will verify later.

For (b), we will prove that $D_1D_2V$ is nonzero. We first write
\begin{equation*}
D_1D_2V=\sum_{j=1}^n\frac{\lambda_jD_1}{x-x_j}\frac{D_2}{x+x_j+U}.
\end{equation*}
For each term in the above sum, note that both fractions are integers, and for every $j\neq1$, $x+x_1+U$ is a factor of $D_2/(x+x_j+U)$. As such, the entire sum is congruent to the first term modulo $x+x_1+U$:
\begin{equation*}
D_1D_2V\equiv\lambda_1\prod_{j=2}^n(x-x_j)\prod_{j=2}^n(x+x_j+U)\mod(x+x_1+U).
\end{equation*}
Each factor of the form $x+x_j+U$ can be further simplified:
\begin{equation*}
x+x_j+U=(x_j-x_1)+(x+x_1+U)\equiv x_j-x_1\mod(x+x_1+U),
\end{equation*}
and so
\begin{equation*}
D_1D_2V\equiv V_1\mod(x+x_1+U),
\end{equation*}
where
\begin{equation*}
V_1:=\lambda_1\prod_{j=2}^n(x-x_j)\prod_{j=2}^n(x_j-x_1).
\end{equation*}
To prove that $D_1D_2V$ is nonzero, it suffices to show that $x+x_1+U$ does not divide $V_1$. To this end, we first note that $V_1$ is nonzero since $x$ and the $x_j$'s are distinct by assumption. Next, since $|\lambda_1|\leq 2m$ and $x$ and each $x_j$ is at most $L$, we have $|V_1|\leq 2mL^{2n-2}\leq 2mL^{4m-2}$, and so it suffices to have
\begin{equation}
\label{eq.b'}
2mL^{4m-2}\leq U
\end{equation}
since $U<x+x_1+U$. We will verify this later.

For (c), we have
\begin{equation*}
|D_2V|\leq\sum_{j=1}^n\frac{|\lambda_j|}{|x-x_j|}\prod_{\substack{j'=1\\j'\neq j}}^n|x+x_{j'}+U|\leq\bigg(\sum_{j=1}^n|\lambda_j|\bigg)\cdot(2L+U)^{n-1}.
\end{equation*}
Considering our assumption on the $\lambda_j$'s we then have $|D_2V|\leq 2m(2L+U)^{n-1}\leq2m(2L+U)^{2m-1}$, and so it suffices to have
\begin{equation}
\label{eq.c'}
2m(2L+U)^{2m-1}<p.
\end{equation}

Overall, we have shown it suffices to have \eqref{eq.a'}, \eqref{eq.b'} and \eqref{eq.c'}.
To satisfy \eqref{eq.b'} for sufficiently large $L$, we take $U:=\lfloor L^{4m-2+\epsilon}\rfloor$.
Then $L=o(U)$, and so $2m(2L+U)^{2m-1}<U^{2m-1+\epsilon}\leq L^{(4m-2+\epsilon)(2m-1+\epsilon)}$ for sufficiently large $U$.
As such, for \eqref{eq.c'}, it suffices to take $L:=\lfloor p^{1/(4m-2+\epsilon)(2m-1+\epsilon)}\rfloor$, which also satisfies \eqref{eq.a'}.
For simplicity, we take $\epsilon=1$.
\qed
\end{proof}

\section{The Main Result}

We are now ready to state the main result of this chapter, which is a generalization of the main result in~\cite{BourgainDFKK:11}.
Later in this section, we will maximize $\epsilon_0$ such that this result implies $\mathrm{ExRIP}[1/2+\epsilon_0]$ with the matrix construction from~\cite{BourgainDFKK:11}.

\begin{theorem}[The BDFKK restricted isometry machine]
\label{theorem.machine}
For every prime $p$, define subsets $\mathcal{A}=\mathcal{A}(p)$ and $\mathcal{B}=\mathcal{B}(p)$ of $\mathbb{F}_p$.
Suppose there exist constants $m\in2\mathbb{N}$, $\ell,\gamma>0$ independent of $p$ such that the following conditions apply:
\begin{description}[(b)]
\item[(a)]
For every sufficiently large $p$, and for every $a\in\mathcal{A}$ and $a_1,\ldots,a_{2m}\in\mathcal{A}\setminus\{a\}$,
\begin{equation*}
\sum_{j=1}^m\frac{1}{a-a_j}=\sum_{j=m+1}^{2m}\frac{1}{a-a_j}
\end{equation*}
only if $(a_1,\ldots,a_m)$ and $(a_{m+1},\ldots,a_{2m})$ are permutations of each other.
Here, division (and addition) is taken in the field $\mathbb{F}_p$.
\item[(b)]
For every $\epsilon>0$, there exists $P=P(\epsilon)$ such that for every $p\geq P$, every subset $S\subseteq\mathcal{B}(p)$ with $|S|\geq p^\ell$ satisfies $E(S,S)\leq p^{-\gamma+\epsilon}|S|^3$.
\end{description}
Pick $\alpha$ such that
\begin{equation}
\label{eq.alpha}
\Omega(p^\alpha)\leq|\mathcal{A}(p)|\leq p^\alpha,
\qquad
|\mathcal{B}(p)|\geq \Omega(p^{1-\alpha+\epsilon'})
\end{equation}
for some $\epsilon'>0$ and every sufficiently large $p$.
Pick $\epsilon_1>0$ for which there exist $\alpha_1,\alpha_2,\epsilon,x,y>0$
such that
\begin{align}
\label{eq.A}
\epsilon_1+\epsilon
&<\alpha_1-\alpha-(4/3)x-\epsilon,\\
\label{eq.B}
\ell
&\leq 1/2+(4/3)x-\alpha_1+\epsilon/2,\\
\label{eq.D}
\epsilon_1+\epsilon
&<\gamma/4-y/4-\epsilon,\\
\label{eq.E}
\alpha_2
&\geq 9x+\epsilon,\\
\label{eq.F.final}
c_0y/8-(\alpha_1/4+9\alpha_2/8)/m
&\leq x/8-\alpha/4,\\
\label{eq.G.final}
\epsilon_1+\epsilon
&<c_0y/8-(\alpha_1/4+9\alpha_2/8)/m,\\
\label{eq.H}
my
&\leq\min\{1/2-\alpha_1,1/2-\alpha_2\},\\
\label{eq.I}
3\alpha_2-2\alpha_1
&\leq(2-c_0)my.
\end{align}
Here, $c_0=1/10430$ is a constant from Proposition~2 in~\cite{BourgainG:11}.
Then for sufficiently large $p$, the $p\times|\mathcal{A}(p)||\mathcal{B}(p)|$ matrix with columns
\begin{equation*}
u_{a,b}:=\frac{1}{\sqrt{p}}\Big(e^{2\pi i(ax^2+bx)/p}\Big)_{x\in\mathbb{F}_p}
\qquad
a\in\mathcal{A},b\in\mathcal{B}
\end{equation*}
satisfies $(p^{1/2+\epsilon_1/2-\epsilon''},\delta)$-RIP for any $\epsilon''>0$ and $\delta<\sqrt{2}-1$, thereby implying $\mathrm{ExRIP}[1/2+\epsilon_1/2]$.
\end{theorem}

Let's briefly discuss the structure of the proof of this result.
As indicated in Section~2, the method is to prove flat-RIP-type cancellations, namely that
\begin{equation}
\label{eq.defn of S}
S(A_1,A_2):=\sum_{\substack{a_1\in A_1\\a_2\in A_2}}\sum_{\substack{b_1\in\Omega_1(a_1)\\b_2\in\Omega_2(a_2)}}\bigg(\frac{a_1-a_2}{p}\bigg)e_p\bigg(\frac{(b_1-b_2)^2}{2(a_1-a_2)}\bigg)
\end{equation}
has size $\leq p^{1-\epsilon_1-\epsilon}$ whenever $\Omega_1$ and $\Omega_2$ are disjoint with size $\leq\sqrt{p}$.
(Actually, we get to assume that these subsets and the $\Omega_i(a_i)$'s satisfy certain size constraints since we have an extra $-\epsilon$ in the power of $p$; having this will imply the general case without the $\epsilon$, as made clear in the proof of Theorem~\ref{theorem.machine}.)
This bound is proved by considering a few different cases.
First, when the $\Omega_i(a_i)$'s are small, \eqref{eq.defn of S} is small by a triangle inequality.
Next, when the $\Omega_i(a_i)$'s are large, then we can apply a triangle inequality over each $A_i$ and appeal to hypothesis~(b) in Theorem~\ref{theorem.machine} and Lemma~\ref{lemma.9}.
However, this will only give sufficient cancellation when the $A_i$'s are small.
In the remaining case, Bourgain et al. prove sufficient cancellation by invoking Lemma~10 in~\cite{BourgainDFKK:11}, which concerns the following quantity:
\begin{equation}
\label{eq.defn of T}
T_{a_1}(A_2,B):=\sum_{\substack{b_1\in B\\a_2\in A_2,~b_2\in\Omega_2(a_2)}}\bigg(\frac{a_1-a_2}{p}\bigg)e_p\bigg(\frac{(b_1-b_2)^2}{4(a_1-a_2)}\bigg).
\end{equation}
Specifically, Lemma~10 in~\cite{BourgainDFKK:11} gives that $|T_{a_1}(A_2,B)|$ is small whenever $B$ has sufficient additive structure.
In the proof of the main result, they take a maximal subset $B_0\subseteq\Omega_1(a_1)$ such that $|T_{a_1}(A_2,B_0)|$ is small, and then they use this lemma to show that $\Omega_1(a_1)\setminus B_0$ necessarily has little additive structure.
By Lemma~\ref{lemma.9}, this in turn forces $|T_{a_1}(A_2,B_1)|$ to be small, and so $|T_{a_1}(A_2,\Omega_1(a_1))|$ (and furthermore $|S(A_1,A_2)|$) are also small due to a triangle inequality.
The reader is encouraged to find more details in the proofs found in Section~5.

What follows is a generalized version of the statement of Lemma~10 in~\cite{BourgainDFKK:11}, which we then use in the hypothesis of a generalized version of Lemma~2 in~\cite{BourgainDFKK:11}:

\begin{definition}
Let $\mathrm{L10}=\mathrm{L10}[\alpha_1,\alpha_2,k_0,k_1,k_2,m,y]$ denote the following statement about subsets $\mathcal{A}=\mathcal{A}(p)$ and $\mathcal{B}=\mathcal{B}(p)$ of $\mathbb{F}_p$ for $p$ prime:

\medskip
\noindent
For every $\epsilon>0$, there exists $P>0$ such that for every $p\geq P$ the following holds:

Take $\Omega_1,\Omega_2\subseteq\mathcal{A}\times\mathcal{B}$ such that
\begin{equation}
\label{eq.6.2}
|A_2|\geq p^{y},
\end{equation}
and for which there exist powers of two $M_1,M_2$ such that
\begin{equation}
\label{eq.2.9a}
\frac{M_i}{2}\leq |\Omega_i(a_i)|<M_i
\end{equation}
and
\begin{equation}
\label{eq.2.9b}
|A_i|M_i\leq 2\sqrt{p}
\end{equation}
for $i=1,2$ and for every $a_i\in A_i$. Then for every $B\subseteq\mathbb{F}_p$ such that
\begin{equation}
\label{eq.6.3}
p^{1/2-\alpha_1}\leq|B|\leq p^{1/2}
\end{equation}
and
\begin{equation}
\label{eq.6.4}
|B-B|\leq p^{\alpha_2}|B|,
\end{equation}
we have that \eqref{eq.defn of T} satisfies
\begin{equation}
\label{eq.6.5}
|T_{a_1}(A_2,B)|\leq|B|p^{1/2-\epsilon_2+\epsilon}
\end{equation}
with $\epsilon_2=k_0y-(k_1\alpha_1+k_2\alpha_2)/m$ for every $a_1\in A_1$.
\end{definition}

\begin{lemma}[generalized version of Lemma~2 in~\cite{BourgainDFKK:11}]
\label{lemma.2}
Take $\mathcal{A}$ arbitrarily and $\mathcal{B}$ satisfying the hypothesis~(b) in Theorem~\ref{theorem.machine}, pick $\alpha$ such that $|\mathcal{A}(p)|\leq p^\alpha$ for every sufficiently large $p$, and pick $\epsilon,\epsilon_1,x>0$ such that $\mathrm{L10}$ holds with \eqref{eq.A}--\eqref{eq.E} and
\begin{equation}
\label{eq.FG}
\epsilon_1+\epsilon
<\epsilon_2
\leq x/8-\alpha/4.
\end{equation}
Then the following holds for every sufficiently large $p$:

Take $\Omega_1,\Omega_2\subseteq\mathcal{A}\times\mathcal{B}$ for which there exist powers of two $M_1,M_2$ such that \eqref{eq.2.9a} and \eqref{eq.2.9b} hold for $i=1,2$ and every $a_i\in A_i$.
Then \eqref{eq.defn of S} satisfies $|S(A_1,A_2)|\leq p^{1-\epsilon_1-\epsilon}$.
\end{lemma}

The following result gives sufficient conditions for $\mathrm{L10}$, and thus Lemma~\ref{lemma.2} above:

\begin{lemma}[generalized version of Lemma~10 in~\cite{BourgainDFKK:11}]
\label{lemma.10}
Suppose $\mathcal{A}$ satisfies hypothesis~(a) in Theorem~\ref{theorem.machine}.
Then $\mathrm{L10}$ is true with $k_0=c_0/8$, $k_1=1/4$ and $k_2=9/8$ provided \eqref{eq.H} and \eqref{eq.I} are satisfied.
\end{lemma}

These lemmas are proved in Section~5.
With these in hand, we are ready to prove the main result:

\begin{proof}[Proof of Theorem~\ref{theorem.machine}]
By Lemma~\ref{lemma.10}, we have that $\mathrm{L10}$ is true with
\begin{equation*}
\epsilon_2=c_0y/8-(\alpha_1/4+9\alpha_2/8)/m.
\end{equation*}
As such, \eqref{eq.F.final} and \eqref{eq.G.final} together imply \eqref{eq.FG}, and so the conclusion of Lemma~\ref{lemma.2} holds.
We will use this conclusion to show that the matrix identified in Theorem~\ref{theorem.machine} satisfies $(p^{1/2},p^{-\epsilon_1})$-weak flat RIP.
Indeed, this will imply $(p^{1/2},p^{-\epsilon_1/2})$-flat RIP by Lemma~\ref{lemma.b}, $(p^{1/2},75p^{-\epsilon_1/2}\log p)$-RIP by Lemma~\ref{lemma.a}, and $(p^{1/2+\epsilon_1/2-\epsilon''},75p^{-\epsilon''}\log p)$-RIP for any $\epsilon''>0$ by Lemma~\ref{lemma.c} (taking $s=p^{\epsilon_1/2-\epsilon''}$).
Since $75p^{-\epsilon''}\log p<\sqrt{2}-1$ for sufficiently large $p$, this will prove the result.

To demonstrate $(p^{1/2},p^{-\epsilon_1})$-weak flat RIP, pick disjoint $\Omega_1,\Omega_2\subseteq\mathcal{A}\times\mathcal{B}$ of size $\leq p^{1/2}$.
We need to show
\begin{equation*}
\bigg|\bigg\langle\sum_{(a_1,b_1)\in\Omega_1}u_{a_1,b_1},\sum_{(a_2,b_2)\in\Omega_2}u_{a_2,b_2}\bigg\rangle\bigg|\leq p^{1/2-\epsilon_1}.
\end{equation*}
Recall that
\begin{align*}
&\sum_{(a_1,b_1)\in\Omega_1}\sum_{(a_2,b_2)\in\Omega_2}\langle u_{a_1,b_1},u_{a_2,b_2}\rangle\\
&\qquad\qquad=\sum_{\substack{a_1\in A_1\\a_2\in A_2}}\sum_{\substack{b_1\in\Omega_1(a_1)\\b_2\in\Omega_2(a_2)}}\langle u_{a_1,b_1},u_{a_2,b_2}\rangle\\
&\qquad\qquad=\sum_{\substack{a_1\in A_1\\a_2\in A_2\setminus A_1}}\sum_{\substack{b_1\in\Omega_1(a_1)\\b_2\in\Omega_2(a_2)}}\frac{\sigma_p}{\sqrt{p}}\bigg(\frac{a_1-a_2}{p}\bigg)e_p\bigg(-\frac{(b_1-b_2)^2}{4(a_1-a_2)}\bigg).
\end{align*}
As such, we may assume that $A_1$ and $A_2$ are disjoint without loss of generality, and it suffices to show that
\begin{equation*}
\bigg|\sum_{\substack{a_1\in A_1\\a_2\in A_2}}\sum_{\substack{b_1\in\Omega_1(a_1)\\b_2\in\Omega_2(a_2)}}\bigg(\frac{a_1-a_2}{p}\bigg)e_p\bigg(-\frac{(b_1-b_2)^2}{4(a_1-a_2)}\bigg)\bigg|\leq p^{1-\epsilon_1}.
\end{equation*}
For each $k$, define the set
\begin{equation*}
A_i^{(k)}
:=\{a_i\in A_i:2^{k-1}\leq|\Omega_i(a_i)|<2^k\}.
\end{equation*}
Then we have
\begin{equation*}
|A_i^{(k)}|2^{k-1}
\leq\sum_{a_i\in A_i^{(k)}}|\Omega_i(a_i)|
=|\{(a_i,b_i)\in\Omega_i:a_i\in A_i^{(k)}\}|
\leq|\Omega_i|
\leq p^{1/2}.
\end{equation*}
As such, taking $M_i=2^k$ gives that $A_i\leftarrow A_i^{(k)}$ satisfies \eqref{eq.2.9a} and \eqref{eq.2.9b}, which enables us to apply the conclusion of Lemma~\ref{lemma.2}.
Indeed, the triangle inequality and Lemma~\ref{lemma.2} together give 
\begin{align*}
\bigg|\sum_{\substack{a_1\in A_1\\a_2\in A_2}}\sum_{\substack{b_1\in\Omega_1(a_1)\\b_2\in\Omega_2(a_2)}}\bigg(\frac{a_1-a_2}{p}\bigg)e_p\bigg(-\frac{(b_1-b_2)^2}{4(a_1-a_2)}\bigg)\bigg|
&\leq\sum_{k_1=1}^{\lceil\frac{1}{2}\log_2 p\rceil}\sum_{k_2=1}^{\lceil\frac{1}{2}\log_2 p\rceil}|S(A_1^{(k_1)},A_2^{(k_2)})|\\
&\leq p^{1-\epsilon_1-\epsilon}\log^2 p\\
&\leq p^{1-\epsilon_1}
\end{align*}
for sufficiently large $p$.
\qed
\end{proof}

To summarize Theorem~\ref{theorem.machine}, we may conclude $\mathrm{ExRIP}[1/2+\epsilon_1/2]$ if we can find
\begin{description}[(iii)]
\item[(i)] $m\in2\mathbb{N}$ satisfying hypothesis~(a),
\item[(ii)] $\ell,\gamma>0$ satisfying hypothesis~(b),
\item[(iii)] $\alpha$ satisfying \eqref{eq.alpha}, and
\item[(iv)] $\alpha_1,\alpha_2,\epsilon,x,y>0$ satisfying \eqref{eq.A}--\eqref{eq.I}.
\end{description}
Since we want to conclude $\mathrm{ExRIP}[z]$ for the largest possible $z$, we are inclined to maximize $\epsilon_1$ subject to (i)--(iv), above.
To find $m,\ell,\gamma,\alpha$ which satisfy (i)--(iii), we must leverage a particular construction of $\mathcal{A}$ and $\mathcal{B}$, and so we turn to Lemma~\ref{lemma.construction of a} and Corollary~\ref{corollary.4}.
Indeed, for any given $m$, Lemma~\ref{lemma.construction of a} constructs $\mathcal{A}$ satisfying hypothesis~(a) such that 
\begin{equation}
\label{eq.alpha of m}
\alpha=1/(2m(4m-1))
\end{equation}
satisfies the first part of \eqref{eq.alpha}.
Next, if we take $\beta:=\alpha-\epsilon'$ and define $r:=\lfloor\beta\log_2p\rfloor$ and $M:=2^{1/\beta-1}$, then \eqref{eq.construction of b} constructs $\mathcal{B}$ which, by Corollary~\ref{corollary.4}, satisfies hypothesis~(b) provided
\begin{equation}
\label{eq.ell gamma range}
(2\tau-1)(\ell-\gamma)\geq10\gamma,
\end{equation}
where $\tau$ is the solution to
\begin{equation*}
\Big(\frac{1}{M}\Big)^{2\tau}+\Big(\frac{M-1}{M}\Big)^{\tau}=1.
\end{equation*}
For this construction, $|\mathcal{B}|=M^r\geq\Omega(p^{1-\beta})$, thereby satisfying the second part of $\eqref{eq.alpha}$.

\begin{figure}[t]
\begin{center}
\includegraphics[width=0.6\textwidth]{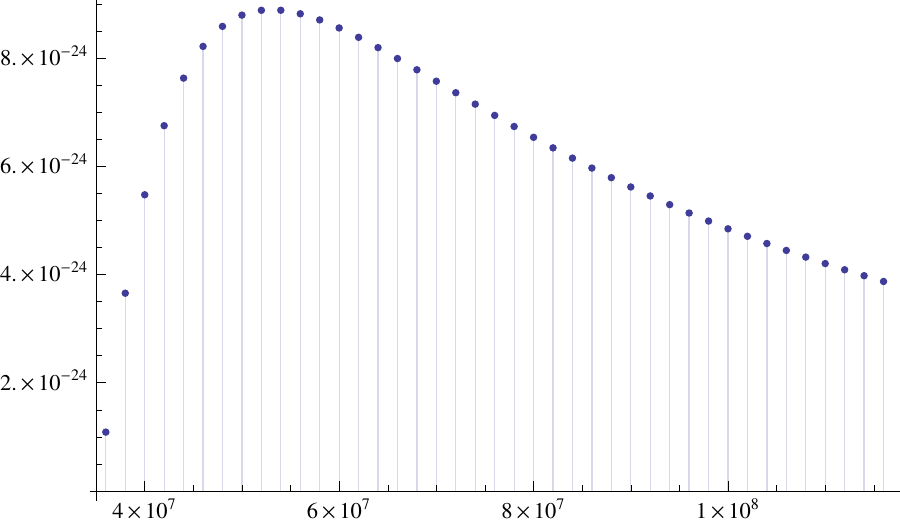}
\end{center}
\caption{The supremum of $\epsilon_1$ as a function of $m$. Taking $\epsilon'=0$, we run a linear program to maximize $\epsilon_1$ subject to the closure of the constraints \eqref{eq.A}--\eqref{eq.I}, \eqref{eq.ell gamma range} for various values of $m$. A locally maximal supremum of $\epsilon_1\approx 8.8933\times 10^{-24}$ appears around $m=53,000,000$.}
\label{fig:1}
\end{figure}

It remains to maximize $\epsilon_1$ for which there exist $m,\epsilon',\ell,\gamma,\alpha_1,\alpha_2,\epsilon,x,y$ satisfying \eqref{eq.A}--\eqref{eq.I}, \eqref{eq.alpha of m} and \eqref{eq.ell gamma range}.
Note that $m$ and $\epsilon'$ determine $\alpha$ and $\tau$, and the remaining constraints \eqref{eq.A}--\eqref{eq.I}, \eqref{eq.ell gamma range} which define the feasible tuples $(\epsilon_1,\ell,\gamma,\alpha_1,\alpha_2,\epsilon,x,y)$ are linear inequalities.
As such, taking the closure of this feasibility region and running a linear program will produce the supremum of $\epsilon_1$ subject to the remaining constraints.
This supremum increases monotonically as $\epsilon'\rightarrow0$, and so we only need to consider the limiting case where $\epsilon'=0$.
Running the linear program for various values of $m$ reveals what appears to be a largest supremum of $\epsilon_1\approx 8.8933\times 10^{-24}$ at $m=53,000,000$ (see Fig.~\ref{fig:1}).
Dividing by $2$ then gives a new record:
\begin{equation*}
\epsilon_0\approx 4.4466\times 10^{-24}.
\end{equation*}
While this optimization makes a substantial improvement (this is over 8,000 times larger than the original record of Bourgain et al.\ in~\cite{BourgainDFKK:11}), the constant is still tiny!
For this particular construction of $\mathcal{A}$ and $\mathcal{B}$, the remaining bottlenecks may lie at the very foundations of additive combinatorics.
For example, if $c_0=1/2$, then taking $m=10,000$ leads to $\epsilon_0$ being on the order of $10^{-12}$.

\section{Proofs of Technical Lemmas}

This section contains the proofs of the technical lemmas (Lemmas~\ref{lemma.2} and~\ref{lemma.10}) which were used to prove the main result (Theorem~\ref{theorem.machine}).

\subsection{Proof of Lemma~\ref{lemma.2}}

First note that $|A_1|M_1<p^{1/2+(4/3)x+\alpha-\alpha_1+\epsilon}$ implies that
\begin{equation*}
|A_1||\Omega_1(a_1)|<p^{1/2+(4/3)x+\alpha-\alpha_1+\epsilon}
\end{equation*}
by \eqref{eq.2.9a}, and by \eqref{eq.2.9b}, we also have
\begin{equation*}
|A_2||\Omega_2(a_2)|<2p^{1/2}.
\end{equation*}
As such, the triangle inequality gives that
\begin{equation*}
|S(A_1,A_2)|\leq|A_1||A_2||\Omega_1(a_1)||\Omega_2(a_2)|\leq 2p^{1+(4/3)x+\alpha-\alpha_1+\epsilon}\leq p^{1-\epsilon_1-\epsilon},
\end{equation*}
where the last step uses \eqref{eq.A}.
Thus, we can assume $|A_1|M_1\geq p^{1/2+(4/3)x+\alpha-\alpha_1+\epsilon}$, and so the assumption $|\mathcal{A}|\leq p^\alpha$ gives
\begin{equation}
\label{eq.6.1}
M_1\geq\frac{1}{|A_1|}p^{1/2+(4/3)x+\alpha-\alpha_1+\epsilon}\geq p^{1/2+(4/3)x-\alpha_1+\epsilon}.
\end{equation}
Applying \eqref{eq.2.9a} and \eqref{eq.6.1} then gives
\begin{equation*}
|\Omega_1(a_1)|\geq\frac{M_1}{2}\geq\frac{1}{2}p^{1/2+(4/3)x-\alpha_1+\epsilon}>p^{1/2+(4/3)x-\alpha_1+\epsilon/2}\geq p^\ell,
\end{equation*}
where the last step uses \eqref{eq.B}.
Note that we can redo all of the preceding analysis by interchanging indices $1$ and $2$.
As such, we also have $|\Omega_2(a_2)|>p^\ell$.
(This will enable us to use hypothesis~(b) in Theorem~\ref{theorem.machine}.)
At this point, we bound
\begin{equation*}
\bigg|\sum_{\substack{b_1\in\Omega_1(a_1)\\b_2\in\Omega_2(a_2)}}e_p\bigg(\frac{(b_1-b_2)^2}{4(a_1-a_2)}\bigg)\bigg|
\end{equation*}
using Lemma~\ref{lemma.9}:
\begin{equation*}
\leq|\Omega_1(a_1)|^{1/2}E(\Omega_1(a_1),\Omega_1(a_1))^{1/8}|\Omega_2(a_2)|^{1/2}E(\Omega_2(a_2),\Omega_2(a_2))^{1/8}p^{1/8}.
\end{equation*}
Next, since $|\Omega_1(a_1)|,|\Omega_2(a_2)|>p^\ell$, hypothesis~(b) in Theorem~\ref{theorem.machine} with $\epsilon\leftarrow4\epsilon$ gives
\begin{equation*}
\leq|\Omega_1(a_1)|^{7/8}|\Omega_2(a_2)|^{7/8}p^{1/8-\gamma/4+\epsilon}.
\end{equation*}
At this point, the triangle inequality gives
\begin{align*}
|S(A_1,A_2)|
&\leq\sum_{\substack{a_1\in A_1\\a_2\in A_2}}\bigg|\sum_{\substack{b_1\in\Omega_1(a_1)\\b_2\in\Omega_2(a_2)}}e_p\bigg(\frac{(b_1-b_2)^2}{4(a_1-a_2)}\bigg)\bigg|\\
&\leq\sum_{\substack{a_1\in A_1\\a_2\in A_2}}|\Omega_1(a_1)|^{7/8}|\Omega_2(a_2)|^{7/8}p^{1/8-\gamma/4+\epsilon}
\end{align*}
which can be further bounded using \eqref{eq.2.9a} and \eqref{eq.2.9b}:
\begin{equation*}
\leq 2^{7/4}|A_1|^{1/8}|A_2|^{1/8}p^{1-\gamma/4+\epsilon}
\end{equation*}
Thus, if $|A_1|,|A_2|<p^y$, then
\begin{equation*}
|S(A_1,A_2)|\leq2^{7/4}p^{1+y/4-\gamma/4+\epsilon}\leq p^{1-\epsilon_1-\epsilon},
\end{equation*}
where the last step uses \eqref{eq.D}.
As such, we may assume that either $|A_1|$ or $|A_2|$ is $\geq p^y$.
Without loss of generality, we assume $|A_2|\geq p^y$.
(Considering \eqref{eq.6.2}, this will enable us to use $\mathrm{L10}$.)

At this point, take $B_0\subseteq\Omega_1(a_1)$ to be a maximal subset satisfying \eqref{eq.6.5} for $B\leftarrow B_0$, and denote $B_1:=\Omega_1(a_1)\setminus B_0$.
Then the triangle inequality gives
\begin{equation*}
|T_{a_1}(A_2,B_1)|\leq\sum_{a_2\in A_2}\bigg|\sum_{\substack{b_1\in B_1\\b_2\in\Omega_2(a_2)}}e_p\bigg(\frac{(b_1-b_2)^2}{4(a_1-a_2)}\bigg)\bigg|,
\end{equation*}
and then Lemma~\ref{lemma.9} gives
\begin{equation*}
\leq\sum_{a_2\in A_2}|B_1|^{1/2}E(B_1,B_1)^{1/8}|\Omega_2(a_2)|^{1/2}E(\Omega_2(a_2),\Omega_2(a_2))^{1/8}p^{1/8}.
\end{equation*}
This can be bounded further by applying $E(\Omega_2(a_2),\Omega_2(a_2))\leq|\Omega_2(a_2)|^3$, \eqref{eq.2.9a}, \eqref{eq.2.9b} and the assumption $|\mathcal{A}|\leq p^\alpha$:
\begin{equation}
\label{eq.*}
\leq 2^{7/8}|B_1|^{1/2}E(B_1,B_1)^{1/8}p^{\alpha/8+9/16}.
\end{equation}
At this point, we claim that $E(B_1,B_1)\leq p^{-x}M_1^3$.
To see this, suppose otherwise.
Then $|B_1|^3\geq E(B_1,B_1)>p^{-x}M_1^3$, implying
\begin{equation}
\label{eq.**}
|B_1|>p^{-x/3}M_1,
\end{equation}
and by \eqref{eq.2.9a}, we also have
\begin{equation*}
E(B_1,B_1)>p^{-x}M_1^3>p^{-x}|\Omega_1(a_1)|^3\geq p^{-x}|B_1|^3.
\end{equation*}
Thus, Lemma~\ref{corollary.1} with $K=p^x$ produces a subset $B_1'\subseteq B_1$ such that
\begin{equation*}
|B_1'|\geq\frac{|B_1|}{20p^x}>\frac{M_1}{20p^{(4/3)x}}\geq\frac{1}{20}p^{1/2-\alpha_1+\epsilon}\geq p^{1/2-\alpha_1}
\end{equation*}
where the second and third inequalities follow from \eqref{eq.**} and \eqref{eq.6.1}, respectively, and
\begin{equation*}
|B_1'-B_1'|\leq 10^7p^{9x}|B_1|\leq p^{9x+\epsilon}|B_1|\leq p^{\alpha_2}|B_1|,
\end{equation*}
where the last step follows from \eqref{eq.E}.
As such, $|B_1'|$ satisfies \eqref{eq.6.3} and \eqref{eq.6.4}, implying that $B\leftarrow B_1'$ satisfies \eqref{eq.6.5} by $\mathrm{L10}$.
By the triangle inequality, $B\leftarrow B_0\cup B_1'$ also satisfies \eqref{eq.6.5}, contradicting $B_0$'s maximality.

We conclude that $E(B_1,B_1)\leq p^{-x}M_1^3$, and continuing \eqref{eq.*} gives
\begin{equation*}
|T_{a_1}(A_2,B_1)|\leq2^{7/8}|B_1|^{1/2}M_1^{3/8}p^{9/16+\alpha/8-x/8}.
\end{equation*}
Now we apply \eqref{eq.6.5} to $B\leftarrow B_0$ and combine with this to get
\begin{align*}
|T_{a_1}(A_2,\Omega_1(a_1))|
&\leq|T_{a_1}(A_2,B_0)|+|T_{a_1}(A_2,B_1)|\\
&\leq|B_0|p^{1/2-\epsilon_1}+2^{7/8}|B_1|^{1/2}M_1^{3/8}p^{9/16+\alpha/8-x/8}.
\end{align*}
Applying $|B_0|,|B_1|\leq|\Omega_1(a_1)|\leq M_1$ by \eqref{eq.2.9a} then gives
\begin{equation*}
|T_{a_1}(A_2,\Omega_1(a_1))|\leq M_1p^{1/2-\epsilon_1}+2^{7/8}M_1^{7/8}p^{9/16+\alpha/8-x/8}.
\end{equation*}
Now we apply the triangle inequality to get
\begin{align*}
|S(A_1,A_2)|
&\leq\sum_{a_1\in A_1}|T_{a_1}(A_2,\Omega_1(a_1))|\\
&\leq|A_1|\Big(M_1p^{1/2-\epsilon_1}+2^{7/8}M_1^{7/8}p^{9/16+\alpha/8-x/8}\Big),
\end{align*}
and applying \eqref{eq.2.9b} and the assumption $|\mathcal{A}|\leq p^\alpha$ then gives
\begin{equation*}
\leq 2p^{1-\epsilon_2}+2^{7/4}p^{1+\alpha/4-x/8}\leq2p^{1-\epsilon_2}+2^{7/4}p^{1-\epsilon_2}\leq p^{1-\epsilon_1-\epsilon},
\end{equation*}
where the last steps use \eqref{eq.FG}.
This completes the proof.

\subsection{Proof of Lemma~\ref{lemma.10}}

We start by following the proof of Lemma~10 in~\cite{BourgainDFKK:11}.
First, Cauchy--Schwarz along with \eqref{eq.2.9a} and \eqref{eq.2.9b} give
\begin{align*}
|T_{a_1}(A_2,B)|^2
&=\bigg|\sum_{\substack{a_2\in A_2\\b_2\in\Omega_2(a_2)}}\bigg(\frac{a_1-a_2}{p}\bigg)\cdot\sum_{b_1\in B}e_p\bigg(\frac{(b_1-b_2)^2}{4(a_1-a_2)}\bigg)\bigg|^2\\
&\leq 2\sqrt{p}\sum_{\substack{a_2\in A_2\\b_2\in\Omega_2(a_2)}}\bigg|\sum_{b_1\in B}e_p\bigg(\frac{(b_1-b_2)^2}{4(a_1-a_2)}\bigg)\bigg|^2.
\end{align*}
Expanding $|w|^2=w\overline{w}$ and applying the triangle inequality then gives
\begin{equation*}
=2\sqrt{p}\bigg|\sum_{\substack{a_2\in A_2\\b_2\in\Omega_2(a_2)}}\sum_{b_1,b\in B}e_p\bigg(\frac{b_1^2-b^2}{4(a_1-a_2)}-\frac{b_2(b_1-b)}{2(a_1-a_2)}\bigg)\bigg|
\leq2\sqrt{p}\sum_{b_1,b\in B}|F(b,b_1)|,
\end{equation*}
where
\begin{equation*}
F(b,b_1)
:=\sum_{\substack{a_2\in A_2\\b_2\in\Omega_2(a_2)}}e_p\bigg(\frac{b_1^2-b^2}{4(a_1-a_2)}-\frac{b_2(b_1-b)}{2(a_1-a_2)}\bigg).
\end{equation*}
Next, H\"{o}lder's inequality $\|F1\|_1\leq\|F\|_m\|1\|_{1-1/m}$ gives
\begin{equation}
\label{eq.6.8}
|T_{a_1}(A_2,B)|^2
\leq2\sqrt{p}|B|^{2-2/m}\bigg(\sum_{b_1,b\in B}|F(b,b_1)|^m\bigg)^{1/m}.
\end{equation}
To bound this, we use a change of variables $x:=b_1+b\in B+B$ and $y:=b_1-b\in B-B$ and sum over more terms:
\begin{equation*}
\sum_{b_1,b\in B}|F(b,b_1)|^m
\leq\sum_{\substack{x\in B+B\\y\in B-B}}\bigg|\sum_{\substack{a_2\in A_2\\b_2\in\Omega_2(a_2)}}e_p\bigg(\frac{xy}{4(a_1-a_2)}-\frac{b_2y}{2(a_1-a_2)}\bigg)\bigg|^m.
\end{equation*}
Expanding $|w|^m=w^{m/2}\overline{w}^{m/2}$ and applying the triangle inequality then gives
\begin{align*}
&=\bigg|\sum_{\substack{x\in B+B\\y\in B-B}}\sum_{\substack{a_2^{(i)}\in A_2\\b_2^{(i)}\in\Omega_2(a_2^{(i)})\\1\leq i\leq m}}\!\!\!\!\!e_p\bigg(\sum_{i=1}^{m/2}\Big[\tfrac{xy}{4(a_1-a_2^{(i)})}-\tfrac{b_2y}{2(a_1-a_2^{(i)})}-\tfrac{xy}{4(a_1-a_2^{(i+m/2)})}+\tfrac{b_2y}{2(a_1-a_2^{(i+m/2)})}\Big]\bigg)\bigg|\\
&\leq\sum_{y\in B-B}\sum_{\substack{a_2^{(i)}\in A_2\\b_2^{(i)}\in\Omega_2(a_2^{(i)})\\1\leq i\leq m}}\bigg|\sum_{x\in B+B}e_p\bigg(\frac{xy}{4}\sum_{i=1}^{m/2}\bigg[\frac{1}{a_1-a_2^{(i)}}-\frac{1}{a_1-a_2^{(i+m/2)}}\bigg]\bigg)\bigg|.
\end{align*}
Next, we apply \eqref{eq.2.9a} to bound the number of $m$-tuples of $b_2^{(i)}$'s for each $m$-tuple of $a_2^{(i)}$'s (there are less than $M_2^m$).
Combining this with the bound above, we know there are complex numbers $\epsilon_{y,\xi}$ of modulus $\leq1$ such that
\begin{equation}
\label{eq.6.9}
\sum_{b_1,b\in B}|F(b,b_1)|^m
\leq M_2^m\sum_{y\in B-B}\sum_{\xi\in\mathbb{F}_p}\lambda(\xi)\epsilon_{y,\xi}\sum_{x\in B+B}e_p(xy\xi/4),
\end{equation}
where
\begin{equation*}
\lambda(\xi)
:=\bigg|\bigg\{a^{(1)},\ldots,a^{(m)}\in A_2:\sum_{i=1}^{m/2}\bigg[\frac{1}{a_1-a^{(i)}}-\frac{1}{a_1-a^{(i+m/2)}}\bigg]=\xi\bigg\}\bigg|.
\end{equation*}
To bound the $\xi=0$ term in \eqref{eq.6.9}, pick $a^{(1)},\ldots,a^{(m)}\in A_2$ such that
\begin{equation}
\label{eq.xi=0}
\sum_{i=1}^{m/2}\bigg[\frac{1}{a_1-a^{(i)}}-\frac{1}{a_1-a^{(i+m/2)}}\bigg]=0.
\end{equation}
Then
\begin{equation*}
\sum_{i=1}^{m/2}\frac{1}{a_1-a^{(i)}}+\frac{m}{2}\cdot\frac{1}{a_1-a^{(1)}}
=\sum_{i=m/2+1}^{m}\frac{1}{a_1-a^{(i)}}+\frac{m}{2}\cdot\frac{1}{a_1-a^{(1)}},
\end{equation*}
and so by hypothesis~(a) in Theorem~\ref{theorem.machine}, we have that $(a^{(1)},\ldots,a^{(m/2)},a^{(1)},\ldots,a^{(1)})$ is a permutation of $(a^{(m/2+1)},\ldots,a^{(m)},a^{(1)},\ldots,a^{(1)})$, which in turn implies that $(a^{(1)},\ldots,a^{(m/2)})$ and $(a^{(m/2+1)},\ldots,a^{(m)})$ are permutations of each other.
Thus, all possible solutions to \eqref{eq.xi=0} are determined by $(a^{(1)},\ldots,a^{(m/2)})$.
There are $|A_2|^{m/2}$ choices for this $m/2$-tuple, and for each choice, there are $(m/2)!$ available permutations for $(a^{(m/2+1)},\ldots,a^{(m)})$.
As such,
\begin{equation}
\label{eq.6.10}
\lambda(0)=(m/2)!|A_2|^{m/2},
\end{equation}
which we will use later to bound the $\xi=0$ term.
In the meantime, we bound the remainder of \eqref{eq.6.9}.
To this end, it is convenient to define the following functions:
\begin{equation*}
\zeta'(z):=\sum_{\substack{y\in B-B\\\xi\in\mathbb{F}_p^*\\y\xi=z}}\epsilon_{y,\xi}\lambda(\xi),
\qquad
\zeta(z):=\sum_{\substack{y\in B-B\\\xi\in\mathbb{F}_p^*\\y\xi=z}}\lambda(\xi).
\end{equation*}
Note that $|\zeta'(z)|\leq\zeta(z)$ by the triangle inequality.
We use the triangle inequality and H\"{o}lder's inequality to bound the $\xi\neq0$ terms in \eqref{eq.6.9}:
\begin{align}
\nonumber
&\bigg|\sum_{y\in B-B}\sum_{\xi\in\mathbb{F}_p^*}\lambda(\xi)\epsilon_{y,\xi}\sum_{x\in B+B}e_p(xy\xi/4)\bigg|\\
\nonumber
&\qquad\qquad=\bigg|\sum_{\substack{x\in B+B\\z\in\mathbb{F}_p}}\zeta'(z)e_p(xz/4)\bigg|\\
\nonumber
&\qquad\qquad\leq\sum_{x\in\mathbb{F}_p}\bigg|1_{B+B}(x)\cdot\sum_{z\in\mathbb{F}_p}\zeta'(z)e_p(xz/4)\bigg|\\
\label{eq.about to convolve}
&\qquad\qquad\leq|B+B|^{3/4}\bigg(\sum_{x\in\mathbb{F}_p}\bigg|\sum_{z\in\mathbb{F}_p}\zeta'(z)e_p(xz/4)\bigg|^4\bigg)^{1/4}.
\end{align}
To proceed, note that
\begin{align*}
\bigg(\sum_{z\in\mathbb{F}_p}\zeta'(z)e_p(xz/4)\bigg)^2
&=\sum_{z,z''\in\mathbb{F}_p}\zeta'(z)\zeta'(z'')e_p(x(z+z'')/4)\\
&=\sum_{z'\in\mathbb{F}_p}(\zeta'*\zeta')(z')e_p(xz'/4),
\end{align*}
where the last step follows from a change of variables $z'=z+z''$.
With this and Parseval's identity, we continue \eqref{eq.about to convolve}:
\begin{align}
\nonumber
&=|B+B|^{3/4}\bigg(\sum_{x\in\mathbb{F}_p}\bigg|\sum_{z'\in\mathbb{F}_p}(\zeta'*\zeta')(z')e_p(xz'/4)\bigg|^2\bigg)^{1/4}\\
\nonumber
&=|B+B|^{3/4}\|\zeta'*\zeta'\|_2^{1/2}p^{1/4}\\
\label{eq.6.11}
&\leq|B+B|^{3/4}\|\zeta*\zeta\|_2^{1/2}p^{1/4},
\end{align}
where the last step follows from the fact that $|(\zeta'*\zeta')(z)|\leq(\zeta*\zeta)(z)$, which can be verified using the triangle inequality.
Since $\zeta(z)=\sum_{\xi\in\mathbb{F}_p^*}1_{B-B}(z/\xi)\lambda(\xi)$, the triangle inequality gives
\begin{align}
\|\zeta*\zeta\|_2
\nonumber
&=\bigg\|\bigg(\sum_{\xi\in\mathbb{F}_p^*}\lambda(\xi)1_{\xi(B-B)}\bigg)*\bigg(\sum_{\xi'\in\mathbb{F}_p^*}\lambda(\xi')1_{\xi'(B-B)}\bigg)\bigg\|_2\\
\nonumber
&\leq\sum_{\xi,\xi'\in\mathbb{F}_p^*}\lambda(\xi)\lambda(\xi')\|1_{\xi(B-B)}*1_{\xi'(B-B)}\|_2\\
\label{eq.6.12}
&=\sum_{\xi,\xi'\in\mathbb{F}_p^*}\lambda(\xi)\lambda(\xi')\|1_{B-B}*1_{(\xi'/\xi)(B-B)}\|_2,
\end{align}
where the last step follows from the (easily derived) fact that $1_{B-B}*1_{(\xi'/\xi)(B-B)}$ is a dilation of $1_{\xi(B-B)}*1_{\xi'(B-B)}$.

To bound \eqref{eq.6.12}, we will appeal to Corollary~2 in~\cite{BourgainDFKK:11}, which says that for any $A\subseteq\mathbb{F}_p$ and probability measure $\lambda$ over $\mathbb{F}_p$,
\begin{equation}
\label{corollary.2}
\sum_{b\in\mathbb{F}_p^*}\lambda(b)\|1_A*1_{bA}\|_2
\ll (\|\lambda\|_2+|A|^{-1/2}+|A|^{1/2}p^{-1/2})^{c_0}|A|^{3/2},
\end{equation}
where $\ll$ is Vinogradov notation; $f\ll g$ means $f=O(g)$.
As such, we need to construct a probability measure and understand its $2$-norm.
To this end, define
\begin{equation}
\label{eq.defn lambda1}
\lambda_1(\xi)
:=\frac{\lambda(\xi)}{\|\lambda\|_1}
=\frac{\lambda(\xi)}{|A_2|^m}.
\end{equation}
The sum $\sum_{\xi\in\mathbb{F}_p}\lambda(\xi)^2$ is precisely the number of solutions to
\begin{equation*}
\frac{1}{a_1-a^{(1)}}+\cdots+\frac{1}{a_1-a^{(m)}}-\frac{1}{a_1-a^{(m+1)}}-\cdots-\frac{1}{a_1-a^{(2m)}}=0,
\end{equation*}
which by hypothesis~(a) in Theorem~\ref{theorem.machine}, only has trivial solutions.
As such, we have
\begin{equation}
\label{eq.6.13}
\|\lambda\|_2^2=m!|A_2|^m.
\end{equation}
At this point, define $\lambda_1'(b)$ to be $\lambda_1(\xi'/b)$ whenever $b\neq0$ and $\lambda_1(0)$ otherwise.
Then $\lambda_1'$ is a probability measure with the same $2$-norm as $\lambda_1$, but it allows us to directly apply \eqref{corollary.2}:
\begin{align}
\nonumber
&\sum_{\xi\in\mathbb{F}_p^*}\lambda_1(\xi)\|1_{B-B}*1_{(\xi'/\xi)(B-B)}\|_2\\
\nonumber
&\qquad\qquad=\sum_{b\in\mathbb{F}_p^*}\lambda_1'(b)\|1_{B-B}*1_{b(B-B)}\|_2\\
\label{eq.6.14}
&\qquad\qquad\ll (\|\lambda_1\|_2+|B-B|^{-1/2}+|B-B|^{1/2}p^{-1/2})^{c_0}|B-B|^{3/2}.
\end{align}

At this point, our proof deviates from the proof of Lemma~10 in~\cite{BourgainDFKK:11}.
By \eqref{eq.defn lambda1}, \eqref{eq.6.13} and \eqref{eq.6.2}, we have
\begin{equation*}
\|\lambda_1\|_2=|A_2|^{-m}\|\lambda\|_2\leq\sqrt{m!}|A_2|^{-m/2}\leq\sqrt{m!}p^{-my/2},
\end{equation*}
Next, \eqref{eq.6.3} and \eqref{eq.6.4} together give
\begin{equation*}
|B-B|\geq|B|\geq p^{1/2-\alpha_1}
\end{equation*}
and
\begin{equation*}
|B-B|\leq p^{\alpha_2}|B|\leq p^{1/2+\alpha_2}.
\end{equation*}
Thus,
\begin{align}
\|\lambda_1\|_2+|B-B|^{-1/2}+|B-B|^{1/2}p^{-1/2}
\nonumber
&\leq \sqrt{m!}p^{-my/2}+p^{\alpha_1/2-1/4}+p^{\alpha_2/2-1/4}\\
\label{eq.bound on c0 part}
&\leq p^{-my/2+4m\epsilon},
\end{align}
where the last step follows from \eqref{eq.H}.
So, by \eqref{eq.6.12}, \eqref{eq.defn lambda1}, \eqref{eq.6.14} and \eqref{eq.bound on c0 part}, we have
\begin{align}
\|\zeta*\zeta\|_2
\nonumber
&\leq|A_2|^{2m}\sum_{\xi'\in\mathbb{F}_p^*}\lambda_1(\xi')\sum_{\xi\in\mathbb{F}_p^*}\lambda_1(\xi)\|1_{B-B}*1_{(\xi'/\xi)(B-B)}\|_2\\
\nonumber
&\ll|A_2|^{2m}(\|\lambda_1\|_2+|B-B|^{-1/2}+|B-B|^{1/2}p^{-1/2})^{c_0}|B-B|^{3/2}\\
\label{eq.bound on zeta*zeta}
&\leq|A_2|^{2m}p^{-(c_0/2)my+4c_0m\epsilon}|B-B|^{3/2},
\end{align}
and subsequent application of \eqref{eq.6.9}, \eqref{eq.6.10}, \eqref{eq.6.11} and \eqref{eq.bound on zeta*zeta} gives
\begin{align}
\sum_{b_1,b\in B}|F(b,b_1)|^m
\nonumber
&\leq(\tfrac{m}{2})!(M_2|A_2|)^m|A_2|^{-m/2}|B-B||B+B|\\
\label{eq.***}
&+O(M_2^m|A_2|^m|B-B|^{3/4}|B+B|^{3/4}p^{-(c_0/4)my+2c_0m\epsilon}p^{1/4}).
\end{align}
By Lemma~4 in~\cite{BourgainDFKK:11} (which states that $|A+A|\leq|A-A|^2/|A|$), condition \eqref{eq.6.4} implies
\begin{equation*}
|B+B|\leq\frac{|B-B|^2}{|B|}\leq p^{2\alpha_2}|B|.
\end{equation*}
We now use this with \eqref{eq.2.9b}, \eqref{eq.6.2} and \eqref{eq.6.4} to bound \eqref{eq.***}:
\begin{equation*}
\ll(\tfrac{m}{2})!(2\sqrt{p})^mp^{-my/2}p^{3\alpha_2}|B|^2+(2\sqrt{p})^mp^{(9/4)\alpha_2}|B|^{3/2}p^{-(c_0/4)my+2c_0m\epsilon}p^{1/4}
\end{equation*}
Next, the left-hand inequality of \eqref{eq.6.3} gives that $p^{1/4}\leq|B|^{1/2}p^{\alpha_1/2}$, leading to the following bound:
\begin{equation*}
\ll|B|^2p^{m/2-my/2+3\alpha_2}+|B|^2p^{m/2+\alpha_1/2+(9/4)\alpha_2-(c_0/4)my+2c_0m\epsilon}.
\end{equation*}
Overall, we have
\begin{equation*}
\displaystyle{\sum_{b_1,b\in B}|F(b,b_1)|^m\leq2^{-m}|B|^2p^{m/2+\alpha_1/2+(9/4)\alpha_2-(c_0/4)my+2m\epsilon}},
\end{equation*}
since $c_0<1$, and $3\alpha_2-2\alpha_1\leq(2-c_0)my$ (i.e., \eqref{eq.I}).
Thus, \eqref{eq.6.8} gives
\begin{align*}
|T(A_2,B)|^2
&\leq\sqrt{p}|B|^{2-2/m}(|B|^2p^{m/2+\alpha_1/2+(9/4)\alpha_2-(c_0/4)my+2m\epsilon})^{1/m}\\
&=|B|^2p^{1-(c_0y/4-\alpha_1/(2m)-(9\alpha_2)/(4m))+2\epsilon}.
\end{align*}
Finally, taking square roots produces the result.

\begin{acknowledgement}
The author was supported by NSF Grant No.\ DMS-1321779.
The views expressed in this chapter are those of the author and do not reflect the official policy or position of the United States Air Force, Department of Defense, or the U.S.\ Government.
\end{acknowledgement}


\begin{thebibliography}{WW}

\bibitem{ApplebaumHSC:09}
L.\ Applebaum, S.\ D.\ Howard, S.\ Searle, R.\ Calderbank,
Chirp sensing codes: Deterministic compressed sensing measurements for fast recovery,
Appl.\ Comput.\ Harmon.\ Anal.\ 26 (2009) 283--290.

\bibitem{BandeiraDMS:13}
A.\ S.\ Bandeira, E.\ Dobriban, D.\ G.\ Mixon, W.\ F.\ Sawin,
Certifying the restricted isometry property is hard,
IEEE Trans.\ Inform.\ Theory 59 (2013) 3448--3450.

\bibitem{BandeiraFMW:13}
A.\ S.\ Bandeira, M.\ Fickus, D.\ G.\ Mixon, P.\ Wong,
The road to deterministic matrices with the restricted isometry property,
J.\ Fourier Anal.\ Appl.\ 19 (2013) 1123--1149.

\bibitem{BaraniukDDW:08}
R.\ Baraniuk, M.\ Davenport, R.\ DeVore, M.\ Wakin,
A simple proof of the restricted isometry property for random matrices,
Constr.\ Approx.\ 28 (2008) 253--263.

\bibitem{BourgainDFKK:11}
J.\ Bourgain, S.\ J.\ Dilworth, K.\ Ford, S.\ Konyagin, D.\ Kutzarova,
Explicit constructions of RIP matrices and related problems,
Duke Math.\ J.\ 159 (2011) 145--185.

\bibitem{BourgainG:09}
J.\ Bourgain, M.\ Z.\ Garaev,
On a variant of sum-product estimates and explicit exponential sum bounds in prime fields,
Math.\ Proc.\ Cambridge Philos.\ Soc.\ 146 (2009) 1--21.

\bibitem{BourgainG:11}
J.\ Bourgain, A.\ Glibichuk,
Exponential sum estimate over subgroup in an arbitrary finite field,
Available online: \url{http://www.math.ias.edu/files/avi/Bourgain_Glibichuk.pdf}

\bibitem{Candes:08}
E.\ J.\ Cand\`{e}s,
The restricted isometry property and its implications for compressed sensing,
C.\ R.\ Acad.\ Sci.\ Paris, Ser.\ I 346 (2008) 589--592.

\bibitem{CasazzaF:06}
P.\ G.\ Casazza, M.\ Fickus,
Fourier transforms of finite chirps,
EURASIP J.\ Appl.\ Signal Process.\ 2006 (2006).

\bibitem{Devore:07}
R.\ A.\ DeVore,
Deterministic constructions of compressed sensing matrices,
J.\ Complexity 23 (2007) 918--925.

\bibitem{FickusMT:12}
M.\ Fickus, D.\ G.\ Mixon, J.\ C.\ Tremain,
Steiner equiangular tight frames,
Linear Algebra Appl.\ 436 (2012) 1014--1027.

\bibitem{FoucartR:13}
S.\ Foucart, H.\ Rauhut,
A Mathematical Introduction to Compressive Sensing,
Berlin, Springer, 2013.

\bibitem{KoiranZ:12}
P.\ Koiran, A.\ Zouzias,
Hidden cliques and the certification of the restricted isometry property,
Available online: arXiv:1211.0665

\bibitem{Mixon:13a}
D.\ G.\ Mixon,
Deterministic RIP matrices: Breaking the square-root bottleneck,
Short, Fat Matrices (weblog)
\url{http://dustingmixon.wordpress.com/2013/12/02/deterministic-rip-matrices-breaking-the-square-root-bottleneck/}

\bibitem{Mixon:13b}
D.\ G.\ Mixon,
Deterministic RIP matrices: Breaking the square-root bottleneck, II,
Short, Fat Matrices (weblog)
\url{http://dustingmixon.wordpress.com/2013/12/11/deterministic-rip-matrices-breaking-the-square-root-bottleneck-ii/}

\bibitem{Mixon:13c}
D.\ G.\ Mixon,
Deterministic RIP matrices: Breaking the square-root bottleneck, III,
Short, Fat Matrices (weblog)
\url{http://dustingmixon.wordpress.com/2014/01/14/deterministic-rip-matrices-breaking-the-square-root-bottleneck-iii/}

\bibitem{Welch:74}
L.\ R.\ Welch,
Lower bounds on the maximum cross correlation of signals,
IEEE Trans.\ Inform.\ Theory 20 (1974) 397--399.

\bibitem{Tao:07}
T.\ Tao,
Open question: deterministic UUP matrices,
What's new (weblog)
\url{http://terrytao.wordpress.com/2007/07/02/open-question-deterministic-uup-matrices/}

\bibitem{TaoV:06}
T.\ Tao, V.\ H.\ Vu,
Additive Combinatorics,
Cambridge U.\ Press, 2006.

\end{thebibliography}
\end{document}